\documentclass{article}
\usepackage{amssymb}
\usepackage{amsmath}

\usepackage{amsthm}
\newtheorem{thm}{Theorem}[section]
\newtheorem{lm}{Lemma}[section]
\begin{document}

\title{Local and Global Properties of $p$-Laplace H\'enon Equation}
\date{}
\author{Geyang, Du\\           
    \footnotesize School of Mathematical Sciences, Peking University, \\
    \footnotesize Beijing 100871, China,
    E-mail\,$:$ gydu@pku.edu.cn \and
    Shulin, Zhou\\     
    \footnotesize LMAM, School of Mathematical Sciences, Peking University,\\
    \footnotesize Beijing 100871, China, E-mail\,$:$ szhou@math.pku.edu.cn}

\maketitle
\begin{abstract}
We first give some apriori estimates of positive radial solutions of $p$-Laplace H\'enon equation. Then we study the local and global properties of those solutions. Finally, we generalize some radial results to the nonradial case.
\end{abstract}

\section{Introduction and Main Results}
In this paper, we study local and global properties of $p$-Laplace H\'enon equation
\begin{equation}\label{original}
-\Delta_p u=|x|^{\alpha}u^q,
\end{equation}
where $\Delta_pu=\mathrm{div}(|\nabla u|^{p-2}\nabla u)$. Local properties refer to local behavior of solutions near a certain point, like removable singularity and the order of isolated singularity. Global properties refer to properties of solutions in $\mathbb{R}^N$.

When $p=2$, this is the usual H\'enon equation
\begin{equation}\label{usual}
-\Delta u=|x|^{\alpha}u^q.
\end{equation}
Equation \eqref{usual} was proposed by astrophysicist H\'enon in \cite{Henon}. The first mathematical study about this equation was by \cite{Ni}. After that, a lot of results such as existence, nonexistence, and symmetry breaking were studied, see \cite{Kenichi,Serra,Smets}.

When $p=2,\alpha=0$, equation \eqref{original} is the well known Lane-Emden equation
\begin{equation*}
-\Delta u=u^q,
\end{equation*}
which was studied in \cite{Chandrasekhar}. Lions studied its isolated singularity in \cite{Lions}.

Because the $p$-Laplacian operator is lack of linearity when $p\neq2$, equation \eqref{original} is more difficult than equation \eqref{usual}. Nevertheless, in \cite{Serrin1} Serrin generalized the Carleson's result about the harmonic function \cite{Carleson} and obtained the well-known local properties of general quasilinear equations. Based on his work, the local and global properties of $p$-Laplace Lane-Emden equation
\begin{equation*}
-\Delta_p u=u^q,
\end{equation*}
were studied in \cite{Bidaut,Laurent}.

Usually, the right hand side of equation \eqref{original} is called a source term, see \cite{Veron}. When the right hand side is changed to negative, like
\begin{equation*}
-\Delta_p u=-|x|^{\alpha}u^q,
\end{equation*}
it is called an absorption term. Results about local and global properties of elliptic partial differential equations with an absorption term can be found in \cite{Brezis1987,Brezis,Friedman,Loewner,Vazquez1980,Veron1981}.

Before stating our main results, we first give some definitions. Define $\mu$ as the fundamental solution of $-\Delta_pu=\delta_0$ in distributional sense,
\begin{equation*}
\mu(x)=\mu(|x|)=\left\{
\begin{array}{cl}
\dfrac{p-1}{N-p}\left(\omega_{N-1}\right)^{\frac{-1}{p-1}}|x|^{\frac{p-N}{p-1}},\ &if\ 1<p<N,\\
\left(\omega_{N-1}\right)^{\frac{-1}{N-1}}\ln\left(\dfrac1{|x|}\right),\ &if\ p=N.
\end{array}
\right.
\end{equation*}
where $\omega_{N-1}$ is the area of unit sphere $S^{N-1}$ and $\delta_0$ is the Dirac delta function.

The concept of continuous solution was introduced by Serrin \cite{Serrin}. $u$ is called a continuous solution of \eqref{original} in $\Omega$ if $u$ is continuous in $\Omega$ with $\nabla u\in L_{loc}^{p}(\Omega)$ and $u$ satisfies
\[\int_{\Omega}|\nabla u|^{p-2}\nabla u\nabla \phi=\int_{\Omega}|x|^{\alpha}u^q\phi,\ \forall \phi\in C_c^{\infty}(\Omega).\]
All the solutions referred in the following are continuous solutions.

This paper is organized as follows. First, in Section 2 we give some lemmas which will be used in the following sections. Then in Section 3 we study the local properties of \eqref{original} in the radial case.
\begin{thm}\label{local}
Assume $1<p<N$, $\Omega$ is an open domain containing $\{0\}$ in $\mathbb{R}^N$, and let $u$ be a positive radial solution of \eqref{original} in $\Omega'=\Omega\backslash\{0\}$.
\begin{enumerate}
\item In the subcritical case $p-1<q<\frac{(N+\alpha)(p-1)}{N-p}$, either $u$ can be extended to $\Omega$ as a $C^1$ solution of \eqref{original} or there exists some constant $C>0,\tilde C>0$ such that $\displaystyle{\lim_{x\to0}\frac {u(x)}{\mu(x)}=C}$. Futhermore, $u$ satisfies $-\Delta_p u-|x|^{\alpha}u^q=\tilde C\delta_0$ in the distributional sense.
\item In the critical case \label{case2} $q=\frac{(N+\alpha)(p-1)}{N-p}$, either $u$ can be extended to $\Omega$ as a $C^1$ solution of \eqref{original} or  \[\lim_{x\to0}|x|^{\frac{N-p}{p-1}}\left(\ln\frac1{|x|}\right)^{\frac{N-p}{(p+\alpha)(p-1)}}u(x)=
    \left[\left(\frac{N-p}{p+\alpha}\right)\left(\frac{N-p}{p-1}\right)^{p-1}\right]^{\frac{N-p}{(p+\alpha)(p-1)}}.\]
\item In the supercritical case \label{case3}
$\frac{(N+\alpha)(p-1)}{N-p}<q<\frac{(N+\alpha)p}{N-p}-1$,
either $u$ can be extended to $\Omega$ as a $C^1$ solution of \eqref{original} or
 \[\lim_{x\to0}|x|^{\frac{p+\alpha}{q+1-p}}u(x)= \lambda=\left[\left(\frac{p+\alpha}{q+1-p}\right)^{p-1}\left(N-\frac{pq+\alpha(p-1)}{q+1-p}\right)\right]^{\frac1{q+1-p}}.\]
\end{enumerate}
\end{thm}

\begin{thm}\label{p=N}
Assume $p=N,q>p-1$, $\Omega$ is an open domain containing $\{0\}$ in $\mathbb{R}^N$, and let $u$ be a positive radial solution of \eqref{original} in $\Omega'=\Omega\backslash\{0\}$.
Then either $u$ can be extended to $\Omega$ as a $C^1$ solution of \eqref{original} or there exists some constant $C>0,\tilde C>0$ such that $\displaystyle{\lim_{x\to0}\frac {u(x)}{\mu(x)}=C}$, and $u$ satisfies $-\Delta_p u-|x|^{\alpha}u^q=\tilde C\delta_0$ in the distributional sense.
\end{thm}

Next, in Section 4 we obtain the following global property by dealing with the exterior problem.
\begin{thm}\label{globalthm}
Assume $1<p<N$ and $p-1<q<\frac{(N+\alpha)p}{N-p}-1$. Then equation \eqref{original} has no positive radial solution in $\mathbb{R}^{N}$.
\end{thm}

Finally, in Section 5 we prove some of the radial results can be generalized to the nonradial case, such as the subcritical case in Theorem \ref{local} and Theorem \ref{globalthm}.
\begin{thm}[Subcritical Case]\label{nonradialsub}
Assume $1<p<N$, $p-1<q<\frac{(N+\alpha)(p-1)}{N-p}$, and let $u$ be a positive solution of \eqref{original} in $\Omega'=\Omega\backslash\{0\}$ satisfying $\dfrac{u(x)}{\mu(x)}\le C$ for some constant $C>0$. Then
\begin{enumerate}
\item either $u$ can be extended to $\Omega$ as a $C^1$ solution of \eqref{original}, or
\item there exists some constant $a>0$ such that $\displaystyle{\lim_{x\to0}}\frac {u(x)}{\mu(x)}=a$.
\end{enumerate}
\end{thm}

\begin{thm}\label{nonradialrn}
Assume $1<p<N$, $p-1<q<\frac{(N+\alpha)p}{N-p}-1$ and let $u(x)$ be a nonnegative solution of \eqref{original} in $\mathbb{R}^N$ satisfying $|x|^{\frac{p+\alpha}{q+1-p}}u(x)\le C$ for some constant $C>0$, then $u\equiv0$.
\end{thm}

As far as we know, there are several papers studying local and global properties of $p$-Laplace Lane-Emden equation, but there are few about $p$-Laplace H\'enon equation. Some results such as Theorem \ref{nonradialrn} are new even for $p$-Laplace Lane-Emden equation.

\section{Preliminaries}
In this section, we list some lemmas which will be used later. For quasilinear partial differential equations
\begin{equation}\label{quasi}
\mathrm{div} A(x,u,\nabla u)=B(x,u,\nabla u),
\end{equation}
where $A$ is a given vector function and $B$ is a measurable function satisfying
\begin{align*}
|A(x,u,\xi)|\le a|\xi|^{p-1}+b|u|^{p-1}+e,\\
|B(x,u,\xi)|\le c|\xi|^{p-1}+d|u|^{p-1}+f,\\
\xi\cdot A(x,u,\xi)\ge|\xi|^{p}-d|u|^{p}-g,
\end{align*}
we assume that $1<p\le N,a\in(0,\infty)$ and coefficients $b$ through $g$ are measurable functions in the respective Lebesgue classes
\[b,e\in L^{\frac{N}{p-1-\epsilon}};c\in L^{\frac{N}{1-\epsilon}}; d,f,g\in L^{\frac{N}{p-\epsilon}},\ \epsilon>0.\]
The following lemma comes from Theorem 1 in \cite{Serrin}.
\begin{lm}[Isolated Singularity]\label{isolated singularity}
Assume $\Omega$ is an open domain containing $\{0\}$, and let $u$ be a continuous solution of \eqref{quasi} in the $\Omega'=\Omega\backslash\{0\}$. Suppose that $u\ge L$ for some constant $L$. Then either $u$ has removable singularity at $0$, or else
\begin{equation*}
u\approx\left\{
\begin{array}{cl}
|x|^{\frac{p-N}{p-1}},\ &if\ p<N,\\
\ln\left(\frac1{|x|}\right),\ &if\ p=N,
\end{array}
\right.
\end{equation*}
in the neighborhood of the origin, where $``\approx"$ means ``has the same order with".
\end{lm}

We call $u$ is a $p$-harmonic function if $\Delta_pu=0$. The following Comparison Principle comes from Theorem 2.15 in \cite{Lindqvist}.
\begin{lm}[Comparison Principle]\label{comparison principle}
Suppose that $u$ and $v$ are p-harmonic functions in a bounded domain $\Omega$. If at each $\zeta\in\partial \Omega$
\[\limsup_{x\to \zeta} u(x)\le \liminf_{x\to \zeta}v(x),\]
excluding the situation $\infty\le\infty$ and $-\infty\le-\infty$, then $u\le v$ in $\Omega$.
\end{lm}

The following Strong Comparison Principle comes from Proposition 1.5.2 in \cite{Veron}.
\begin{lm}[Strong Comparison Principle]\label{strong comparison principle}
Let $\Omega\subset \mathbb{R}^N$ be a domain, $p>1$ and $c\in L^{\infty}_{loc}(\Omega)$. Assume $u$ and $v$ belong to $C^1(\Omega)$, satisfy
\[-\Delta_p u+cu\le0\ and\ -\Delta_p v+cv\ge0\ in\ \mathcal{D}'(\Omega)\]
and $\nabla v$ never vanishes in $\Omega$. If $u\le v$ in $\Omega$ and there exists $x_0\in \Omega$ such that $u(x_0)=v(x_0)$, then $u\equiv v$ in $\Omega$.
\end{lm}


\section{Local Properties in the Radial Case}
We  first give some apriori estimates which will be used in the proof of Theorem \ref{local}.
\begin{lm}[Apriori Estimate]\label{apriori}
Assume $1<p<N$, $\Omega$ is an open domain containing $\{0\}$, and let $u$ be a positive radial solution of \eqref{original} in $\Omega'=\Omega\backslash\{0\}$.
\begin{enumerate}
\item If $p-1<q$, then there exists some constant $C>0$ such that $\dfrac{u(x)}{\mu(x)}\le C$ near 0.
\item If $q=\frac{(N+\alpha)(p-1)}{N-p}$, then  there exists some constant $C>0$ such that\\ $\dfrac{u(x)}{\mu(x)}\cdot\left(\ln\frac1{|x|}\right)^{\frac{N-p}{(p-1)(p+\alpha)}}\le C$ near 0.
\item If $q>\frac{(N+\alpha)(p-1)}{N-p}$, then  there exists some constant $C>0$ such that $|x|^{\frac{p+\alpha}{q+1-p}}u(x)\le C$ near 0.
\end{enumerate}
\end{lm}

\begin{proof}

The idea of proof comes from \cite{Laurent}. Without loss of generality, we assume $\overline{B_1(0)}\subset\Omega$. Let
    \[\beta=\frac{p-N}{p-1}, s=r^{\beta}, v(s)=u(r),\]
then we obtain
\begin{equation}\label{apriorieqn}
(p-1)|v'|^{p-2}v''+s^{\frac{1-\beta}{\beta}p+\frac{\alpha}{\beta}}\frac{v^q}{|\beta|^p}=0,\ s\in[1,\infty)
\end{equation}
As a result, $v''<0$, which means $v'$ is decreasing and bounded in $[1,\infty)$.
\begin{enumerate}
\item When $v(s)$ is bounded, we can automatically get the conclusion. So we only need to consider the unbounded case, which means $v(s)\to \infty$ as $s\to \infty$ because of the concavity of $v$. It is easy to check $v'\ge0$. According to L'Hospital's rule, $\displaystyle{\lim_{s\to\infty}\frac{v(s)}{s}=\lim_{s\to\infty}v'(s)}$. So $\frac{v(s)}{s}$ is bounded when s is large, which means $\frac{u(x)}{\mu(x)}$ is bounded near 0. This proves the first part.\\
\item According to the Mean Value Theorem,
\[\frac{v(s)-v(1)}{s-1}=v'(\theta)\ge v'(s), 1<\theta<s.\]
 So $v(s)\ge sv'(s)+o(s)$ when $s$ is large. As a result, we induce from \eqref{apriorieqn} that there exists some $c>0$ such that when $s$ is large
\[ \left((v')^{p-1}\right)'+cs^{\frac{1-\beta}{\beta}p+\frac{\alpha}{\beta}+q}(v')^q\le0.\]
   Let $\gamma=\frac{1-\beta}{\beta}p+\frac{\alpha}{\beta}+q,\ \psi(s)=(v'(s))^{p-1}$. We study the following two cases in the situation when s is large.

 When $q=\frac{(N+\alpha)(p-1)}{N-p}$, $\gamma=-1$, $\psi(s)$ satisfies
\[\psi'+cs^{-1}\psi^{\frac{q}{p-1}}\le0.\]
After integration, we get $\psi(s)\le c(\ln s)^{-\frac{p-1}{q+1-p}}$, which means
\[v'(s)\le c(\ln s)^{-\frac{1}{q+1-p}}=c(\ln s)^{-\frac{N-p}{(p+\alpha)(p-1)}}.\]
As a consequence,
$\frac{v(s)}{s}(\ln s)^{\frac{N-p}{(p+\alpha)(p-1)}}$ is bounded when s is large. This means
$\frac{u(x)}{\mu(x)}\left(\ln \frac1{|x|}\right)^{\frac{N-p}{(p+\alpha)(p-1)}}$ is bounded near 0.

\item When $q>\frac{(N+\alpha)(p-1)}{N-p}$, $\gamma>-1$, $\psi(s)$ satisfies
\[\psi'+cs^{\gamma}\psi^{\frac{q}{p-1}}\le0.\]
After integration, we get $\psi(s)\le cs^{-\frac{\gamma+1}{q+1-p}(p-1)}$, which means
\[v'(s)\le cs^{-\frac{\gamma+1}{q+1-p}}=cs^{-\frac{p+\alpha}{\beta(q+1-p)}-1}.\]
As a consequence, $v(s)s^{\frac{p+\alpha}{\beta(q+1-p)}}$ is bounded when s is large. This means
$u(x)|x|^{\frac{p+\alpha}{q+1-p}}$ is bounded near 0.
\end{enumerate}
\end{proof}

Now we are ready to prove Theorem \ref{local}.
\begin{proof}
\begin{enumerate}
\item\textbf{the Subcritical Case}\\
$|x|^{\alpha}u^q=|x|^{\alpha}u^{q+1-p}u^{p-1}$. From Lemma \ref{apriori}, $\frac{u}{\mu}\le C$ near 0. In addition, when $p-1<q<\frac{(N+\alpha)(p-1)}{N-p}$, we have $q+1-p<\frac{(p+\alpha)(p-1)}{N-p}$. So \[|x|^{\alpha}u^{q+1-p}\in L^{\frac Np+\epsilon}(\Omega)\]
 for some $\epsilon>0$.
According to Lemma \ref{isolated singularity}, either $u$ can be extended to $\Omega$ as a $C^1$ solution of \eqref{original} or there exists constants $c_1,c_2>0$ such that $c_1\mu(x)\le u(x)\le c_2\mu(x)$ near 0. Using the notation in Lemma \ref{apriori}, we have $\frac{v(s)}{s}\ge c_1$ which means $v'(s)\ge c_1$ when s is large. From the proof of Lemma \ref{apriori}, we know $v'(s)$ is decreasing in $[1,\infty)$. So there exists a constant $C$ such that $\displaystyle{\lim_{s\to \infty}\frac{v(s)}{s}=\lim_{s\to \infty} v'(s)=C}$, which means \[\lim_{x\to 0}\frac{u(x)}{\mu(x)}=C,\lim_{x\to 0}\frac{u_r(r)}{\mu_r(r)}=\tilde C,\]
where $\tilde C=(C\cdot\frac{N-p}{p-1})^{p-1}\cdot\omega_{N-1}$.

For $\forall \varphi\in C^\infty_c(B_1(0))$, if we multiply both sides of \eqref{original} by $\varphi$ and integrate by parts over $B_1(0)\backslash B_\epsilon(0)$, then we get
\begin{equation}\label{epsilon}
\int_{\partial B_\epsilon(0)}|\nabla u|^{p-2}\nabla u\cdot \vec n \varphi+\int_{B_1(0)\backslash B_\epsilon(0)}|\nabla u|^{p-2}\nabla u\cdot \nabla \varphi=\int_{B_1(0)\backslash B_\epsilon(0)}|x|^\alpha u^q\varphi,
\end{equation}
where $\vec n=\frac{|x|}{x}$. It is known that the fundamental solution $\mu(x)$ satisfies
\[-\int_{\partial B_\epsilon(0)}|\nabla \mu|^{p-2}\nabla \mu\cdot \vec n \varphi+\int_{B_\epsilon(0)}|\nabla \mu|^{p-2}\nabla \mu\cdot \nabla \varphi=\varphi(0).\]
Sending $\epsilon \to 0$ in \eqref{epsilon}, we have
\[-\tilde C\varphi(0)+\int_{B_1(0)}|\nabla u|^{p-2}\nabla u\cdot \nabla \varphi=\int_{B_1(0)}|x|^\alpha u^q\varphi,\]
which implies
\[-\Delta_p u-|x|^{\alpha}u^q=\tilde C\delta_0\]
in the distributional sense.\\
\item\textbf{the Critical Case}\\
Let
\[w(s)=r^{\delta}u(r),s=-\ln r,\delta=\frac{N-p}{p-1}=\frac{p+\alpha}{q+1-p}.\]
Then \eqref{original} can be transformed into
\[[|w'+\delta w|^{p-2}(w'+\delta w)]'+w^q=0.\]
As a result,
\[\left.|w'+\delta w|^{p-2}(w'+\delta w)\right|_t^{t_n}=-\int_t^{t_n}w^q.\]
From Lemma \ref{apriori}, $w(s)\le C s^{\frac{p-N}{(p+\alpha)(p-1)}}$, so for $\forall t>0, w^q\in L^1(t,\infty)$. As a consequence,
\[|w'(t)+\delta w(t)|^{p-2}(w'(t)+\delta w(t))=\int_t^{\infty}w^q.\]
This means $w'(t)+\delta w(t)>0, t\in (0,\infty)$,
\[(w'(t)+\delta w(t))^{p-1}=\int_t^{\infty}w^q.\]
Using the same procedure as in \cite{Laurent} in the critical case, we claim that
\begin{enumerate}
\item $w(t)$ is decreasing in $(1,\infty)$,
\item $\dfrac{w'(t)}{w(t)}\to0$, or $\dfrac{w'(t)}{w(t)}\to -\delta$ as $t\to \infty$.
\end{enumerate}
If $\frac{w'(t)}{w(t)}\to -\delta$ as $t\to \infty$, we can choose $\epsilon_0$ small such that $w(t)\le ce^{(-\delta+\epsilon_0)t}$ for some $c>0$ when $t$ is large, which means $u(r)\le c\dfrac{1}{r^\epsilon_0}$ as $r\to 0$. As in the subcritical case, we can prove $|x|^{\alpha}u^{q+1-p}\in L^{\frac Np+\epsilon}$ for some $\epsilon>0$. So by Lemma \ref{isolated singularity}, $u$ is regular or have the same order with $\mu(x)$. The later is impossible.\\
If $\frac{w'(t)}{w(t)}\to 0$ as $t\to \infty$, we have when $t$ is large,
\[(\delta w(t))^{p-1}\left[1+o\left(\frac{w'(t)}{w(t)}\right)\right]=\int_t^{\infty}w^q.\]
So
\begin{align*}
w(s)\to\left[\frac{p-1}{q+1-p}\left(\frac1{\delta^{p-1}}s-c\right)^{-1}\right]^{\frac{N-p}{(p+\alpha)(p-1)}},\\ s^{\frac{N-p}{(p+\alpha)(p-1)}}w(s)\to \left[\left(\frac{N-p}{p+\alpha}\right)\left(\frac{N-p}{p-1}\right)^{p-1}\right]^{\frac{N-p}{(p+\alpha)(p-1)}}
\end{align*}
as $s\to \infty$.\\

\item\textbf{the Supercritical Case}\\
As in the proof of the critical case, we set $\delta=\dfrac{p+\alpha}{q+1-p}, w(s)=r^\delta u(r)$ where $s=-\ln r$. Then equation \eqref{original} becomes
\begin{equation}\label{supercase}
|w'+\delta w|^{p-2}(a_1 w+a_2 w'+a_3 w'')+w^q=0
\end{equation}
where
\[a_1=\delta [(p-1)(\delta+1)-(N-1)]=\delta c_1, a_2=(p-1)\delta+c_1, a_3=p-1.\]
It is the same equation as obtained in \cite{Laurent} except that $\delta$ is different. Regarding \eqref{supercase} as an autonomous system with variables $(w(s),w_s(s))$, we can prove $(\lambda,0)$ is asymptotically stable when $\frac{(N+\alpha)(p-1)}{N-p}<q<\frac{(N+\alpha)p}{N-p}-1$. By the phase plane analysis, as in \cite{Laurent} we assert either $(w(s),w_s(s))\to (\lambda,0)$ which means $\lim_{x\to0}|x|^{\frac{p}{q+1-p}}u(x)=\lambda$, or $w(s)\le ce^{-\delta s}$ which means $u(x)$ is regular. This completes the proof.
\end{enumerate}
\end{proof}

Next we prove Theorem \ref{p=N}.
\begin{proof}
Let $s=-\ln r, v(s)=u(r)$, then
\[(N-1)|v'|^{N-2}v''+e^{-s(N+\alpha)}v^q=0\]
Using the same procedure as in the proof of the subcritical case when $p<N$, we can obtain the conclusion.
\end{proof}

\section{Global Properties in the Radial Case}
\begin{lm}[Exterior Problem]\label{exteriorradial}
Assume $1<p<N$, and $G=\{x\in \mathbb{R}^N:|x|\ge1\}$.
\begin{enumerate}
\item If $p-1<q\le\frac{(N+\alpha)(p-1)}{N-p}$, then \eqref{original} has no positive radial solution in $G$.\\
\item If $\frac{(N+\alpha)(p-1)}{N-p}<q$, and $u$ is a positive radial solution of \eqref{original} in $G$, then there exists some constants $c_1,c_2>0$ such that
    \[c_1\mu(x)\le u(x)\le c_2|x|^{-\frac{p+\alpha}{q+1-p}}\]
     when $|x|$ is large.
\end{enumerate}
\end{lm}

\begin{proof}
\begin{enumerate}
\item We argue by contradiction, assuming that equation \eqref{original} has a positive radial solution $u$ in $G$. Using the same transformation as in the proof of Lemma \ref{apriori},
\[v(s)=u(r),s=r^\beta,\beta=\frac{p-N}{p-1},\]
we obtain
\[(|v'|^{p-2}v')'+\frac1{|\beta|^p}s^{\frac1\beta(p+\alpha)-p}v^q=0,\ s\in(0,1].\]
Because $v''(s)<0$, there exists a constant $c$ such that $v(s)\to c$ as $s\to 0$. As in \cite{Laurent}, we can prove $c=0$ and when $s$ is small
\begin{equation*}
v'(s)\le\left\{\begin{array}{cl}
cs^{-\frac1\beta\frac{p+\alpha}{q+1-p}-1},&q<\frac{(N+\alpha)(p-1)}{N-p},\\
c\left(\ln\frac1s\right)^{-\frac1{q+1-p}},&q=\frac{(N+\alpha)(p-1)}{N-p}.
\end{array}\right.
\end{equation*}
This means $v'(s)\to 0$ as $s\to 0$. It is a contradiction to $v'(0)>0$.
\item Because $v(s)$ is concave, $\frac{v(s)}{s}\ge v'(s)$. In addition, $v'(s)>c_1>0$ when $s$ is small, so $\frac{v(s)}{s}\ge c_1$ which proves $u(x)\ge c_1\mu(x)$.

    When $q>\frac{(N+\alpha)(p-1)}{N-p}$, and $s$ is small, there exists constant $c_2>0$ such that $v'(s)\le c_2s^{-\frac1\beta\frac{p+\alpha}{q+1-p}-1}$. So
    \[\frac{v(s)}{s}\le 2c_2s^{-\frac1\beta\frac{p+\alpha}{q+1-p}-1},\]
    which means \[u(x)\le 2c_2|x|^{-\frac{p+\alpha}{q+1-p}}.\]
    This completes the proof.
\end{enumerate}
\end{proof}

\textbf{Proof of Theorem \ref{globalthm}}
\begin{proof}
We argue by contradiction, assuming that $u$ is a positive radial solution of \eqref{original} in $\mathbb{R}^N$. By Lemma \ref{exteriorradial} we only need to prove the case when $\frac{(N+\alpha)(p-1)}{N-p}<q<\frac{N+\alpha}{N-p}-1$. Multiply both sides of \eqref{original} by $u$ and integrate by parts in $B_R=\{x\in \mathbb{R}^N:|x|\le R\}$,
\begin{align*}
\int_{\partial B_R}|\nabla u|^{p-2}\nabla u\cdot \vec n uds+\int_{B_R}|\nabla u|^pdx=\int_{B_R}|x|^\alpha u^{q+1}dx.
\end{align*}
Multiply both sides of \eqref{original} by $\nabla u\cdot x$ and integrate by parts in $B_R$,
\begin{align}\label{boundary}
&\left(\frac Np-1-\frac{N+\alpha}{q+1}\right)\int_{B_R}|x|^\alpha u^{q+1}dx+\left(\frac Np-1\right)\int_{\partial B_R}|\nabla u|^{p-2}\nabla u\cdot \vec n uds\notag\\
+&\int_{\partial B_R}|\nabla u|^{p-2}\nabla u\cdot \vec n \nabla u\cdot xds-\frac1p\int_{\partial B_R}|\nabla u|^p (x\cdot\vec n)ds\\
+&\frac{1}{q+1}\int_{\partial B_R}|x|^\alpha u^{q+1}(x\cdot\vec n)ds=0\notag
\end{align}
According to Lemma \ref{exteriorradial}, $v'\le2c_2s^{-\frac1\beta\frac{p+\alpha}{q+1-p}-1}$, so $|\nabla u(R)|\le CR^{-\frac{q+1+\alpha}{q+1-p}}$ when $R$ is large.
Because $u$ is positive in $\mathbb{R}^N$, the first part of \eqref{boundary} is negative when $q<\frac{(N+\alpha)p}{N-p}-1$. The boundary parts has the same order with $R^{N-p-\frac{(p+\alpha)p}{q+1-p}}$, so they tend to $0$ as $R\to \infty$. This is a contradiction, which implies \eqref{original} has no positive radial solution in $\mathbb{R}^N$ when $p-1<q<\frac{(N+\alpha)p}{N-p}-1$.
\end{proof}

\section{Nonradial Case}
\subsection{Proof of Theorem \ref{nonradialsub}}
\begin{proof}
Because $\dfrac{u(x)}{\mu(x)}\le C$, by Lemma \ref{isolated singularity}, either $u$ can be extended to $\Omega$ as a $C^1$ solution of \eqref{original} or there exists constants $c_1,c_2>0$ such that
\[c_1\mu(x)\le u(x)\le c_2\mu(x).\]
Assume $\displaystyle{\limsup_{x\to0}\frac{u(x)}{\mu(x)}=a}$, then there exists a sequence $\{x_n\}$ satisfying \[\lim_{n\to \infty}x_n=0,\ \lim_{n\to \infty}\frac{u(x_n)}{\mu(x_n)}=a.\]
Denote $x_n=r_n\xi_n$, where $r_n=|x_n|$. Define $u_{r_n}(\xi)=\frac{u(r_n\xi)}{\mu(r_n)},\ \xi\in\left(0,\frac1{r_n}\right)$. It can be easily checked that $u_{r_n}$ satisfies
\begin{equation}\label{u_r}
div_\xi(|\nabla_\xi u_{r_n}|^{p-2}\nabla_\xi u_{r_n})+r_n^{p+\alpha}\mu(r_n)^{q+1-p}|\xi|^{\alpha}u_{r_n}^q=0
\end{equation}
As $q<\frac{(N+\alpha)(p-1)}{N-p}$, and $q+1-p<\frac{(p+\alpha)(p-1)}{N-p}$, so \[r_n^{p+\alpha}\mu(r_n)^{q+1-p}=\mu(1)r_n^{(1-\theta)(p+\alpha)}\]
 for some $0<\theta<1$. In addition, according to definition, $u_{r_n}(\xi)\le C\frac{\mu(\xi)}{\mu(1)}$. So $u_{r_n}$ is bounded in $\left[\frac12,\frac32\right]$, and
 \[r_n^{p+\alpha}\mu(r_n)^{q+1-p}u_{r_n}^q\to 0\]
 as $n\to \infty$. By \cite{Lieberman,Tolksdorf},  $u_{r_n}$ is H\"older continuous in $\left[\frac12,\frac32\right]$. As a result, by Arzela-Ascoli theorem, there exists a subsequence still denoted by $u_{r_n}(\xi)$ such that $u_{r_n}(\xi)\to v(\xi)$ in $\left[\frac12,\frac32\right]$ for some $p$-harmonic function $v(\xi)$. Clearly, \[v(\xi)\le a\frac{\mu(\xi)}{\mu(1)},\ v(\xi_n)=a\frac{\mu(\xi_n)}{\mu(1)},\]
 so by Lemma \ref{strong comparison principle}, $v(\xi)=a\frac{\mu(\xi)}{\mu(1)}$. For any $\epsilon>0$ we can choose $n$ large such that
 \[u_{r_n}(\xi)\ge (a-\epsilon) \frac{\mu(\xi)}{\mu(1)}\]
 when $\epsilon$ is small, that is
 \[u(r_n\xi)\ge(a-\epsilon)\mu(r_n\xi).\] In particular, $u(x)\ge(a-\epsilon)\mu(x)$ on the boundary of annuli $r_{n+1}\le|x|\le r_n$, by Lemma \ref{comparison principle}, $u(x)\ge(a-\epsilon)\mu(x)$ in $r_{n+1}\le|x|\le r_n$. Letting $n\to \infty$, we have $u(x)\ge(a-\epsilon)\mu(x)$ in $B_{r_{n_0}}\backslash\{0\}$ for some big $n_0$. As a consequence, $\displaystyle\liminf_{x\to 0}\frac{u(x)}{\mu(x)}\ge a-\epsilon $. Letting $\epsilon \to 0$, we complete the proof.
\end{proof}
Next, we use the method in \cite{Bidaut} to get nonexistence result in nonradial case.

\begin{thm}\label{exteriornonradial}
If $1<p<N$, and $p-1<q\le\frac{(N+\alpha)(p-1)}{N-p}$, then \eqref{original} has no positive solution in $G$, where $G=\{x\in \mathbb{R}^N:|x|\ge1\}$.
\end{thm}
\begin{proof}
We argue by contradiction, assuming that \eqref{original} has a positive solution in $G$. Denote $m=\min_{|x|=2} u(x)$, define a sequence of radial functions as follows: $u_{n,0}\equiv0$, when $k\ge1$
\begin{equation*}
\left\{\begin{array}{cc}
-\Delta_p u_{n,k}=|x|^{\alpha}u^q_{n,k-1},&2<|x|<n,\\
u_{n,k}=m,&|x|=2,\\
u_{n,k}=0,&|x|=n.
\end{array}\right.
\end{equation*}
Using Lemma \ref{comparison principle}, we can prove $u_{n,k-1}(x)\le u_{n,k}(x)\le u(x)$. In addition, by \cite{Lieberman,Tolksdorf} $u_{n,k}(x)$ is H\"older continuous. So there exists a subsequence still denoted by $u_{n,k}(x)$ such that $u_{n,k}(x)\to u_n(x)$ as $k\to \infty$ for some positive radial function $u_n(x)$ satisfying
\begin{equation*}
\left\{\begin{array}{cc}
-\Delta_p u_n=|x|^{\alpha}u^q_n,&2<|x|<n,\\
u_n=m,&|x|=2,\\
u_n=0,&|x|=n.
\end{array}\right.
\end{equation*}
$u_n$ is increasing with $n$ in each common domain, and $u_n\le u(x)$, by the diagonal method, there exists a positive radial function $v(x)$ such that $u_n(x)\to v(x)$ in $|x|\ge2$. In addition, $v(x)$ satisfies
\begin{equation*}
\left\{\begin{array}{cc}
-\Delta_p v=|x|^{\alpha}v^q,&|x|>2,\\
v=m,&|x|=2.
\end{array}\right.
\end{equation*}
This is a contradiction to Theorem \ref{exteriorradial} which implies equation \eqref{original} has no positive radial solution in $|x|\ge 2$.
\end{proof}

\subsection{Proof of Theorem \ref{nonradialrn}}
\begin{proof}
The idea comes from Lemma 1.1 in \cite{Friedman}. Let $\phi_b(\xi)=b^{\frac{p+\alpha}{q+1-p}}u(b\xi)$, where $b$ is chosen such that $|x|<b$, then $\phi_b(\xi)$ satisfies
\[div_{\xi}\left(|\nabla_{\xi}\phi_b|^{p-2}\nabla_{\xi}\phi_b\right)+|\xi|^{\alpha}\phi_b^q=0.\]
Let $\Gamma=\{1\le |\xi| \le 4\}$, $\Gamma^*=\{2\le |\xi| \le 3\}$. From  $|x|^{\frac{p+\alpha}{q+1-p}}u(x)\le C$, we get
\[||\phi_b||_{L^\infty(\Gamma)}\le b^{\frac{p+\alpha}{q+1-p}}|b\xi|^{-\frac{p+\alpha}{q+1-p}}\le C,\]
so $||\nabla\phi_b||_{C^\alpha(\Gamma^*)}\le C$, which means
\[||\nabla u||_{C^\alpha(\Gamma^*)}\le Cb^{-\frac{q+1+\alpha}{q+1-p}}\le  C|x|^{-\frac{q+1+\alpha}{q+1-p}}.\]
As we have the gradient estimate of $u$, we can use the same procedure as in the proof of Theorem \ref{globalthm} to prove $u\equiv0$.
\end{proof}
\section*{Acknowledgements}
This work was supported by the National Natural Science Foundation of China (Grant No. 12071009).

\end{document}